\documentclass[12pt]{article}

\usepackage{amssymb,amsmath,latexsym,graphics, supertabular}

\usepackage[enableskew]{youngtab}

\usepackage{amsbsy}

\usepackage{amscd}

\usepackage{verbatim}

\usepackage{amsfonts}

\usepackage{amsthm}

\usepackage{times}

\usepackage{mathrsfs}

\usepackage{verbatim}

\usepackage{graphicx,subfigure}

\usepackage[colorlinks]{hyperref}

\usepackage{xypic}

\usepackage{txfonts}

\usepackage[english]{babel}

\usepackage[T1]{fontenc}

\usepackage[utf8]{inputenc}

\usepackage{tikz}

\usetikzlibrary{backgrounds,fit,decorations.pathreplacing}
\usetikzlibrary{arrows,shapes,positioning}
\usetikzlibrary{calc,through,backgrounds,chains}

\usetikzlibrary{arrows,shapes,snakes,automata,backgrounds,petri}

\newenvironment{pf*}[1]{\proof[#1]}{\endproof}
\newtheorem{Theorem}[equation]{Theorem}

\newtheorem{Lemma}[equation]{Lemma}

\theoremstyle{definition}

\newtheorem{Example}[equation]{Example}

\newtheorem{Conjecture}[equation]{Conjecture}

\theoremstyle{remark}

\newtheorem{Remark}[equation]{Remark}

\numberwithin{equation}{section}

\newcommand{\PP}{{\mathbb P}}
\newcommand{\C}{{\mathbb C}}
\newcommand{\Z}{{\mathbb Z}}

\newcommand{\R}{{\mathbb R}}

\newcommand{\mc}[1]{\mathcal{#1}}
\newcommand{\ms}[1]{\mathscr{#1}}

\newcommand{\beq}{\begin{equation}}
\newcommand{\eeq}{\end{equation}}

\newcommand{\mt}[1]{\text{#1}}
\newcommand{\ud}{\,\text{d}}

\def\vi{\varepsilon}

\DeclareMathOperator{\Des}{Des}

\def\ds#1{\displaystyle{#1}}

\begin{document}

\title{Ordered Bell numbers, Hermite polynomials, skew Young tableaux, and Borel orbits}
\author{Mahir Bilen Can\\ Michael Joyce}
\maketitle

\begin{abstract}
We give three interpretations of the number $b$ of orbits of the Borel subgroup of upper triangular matrices
on the variety $\ms{X}$ of complete quadrics.  First, we show that $b$ is equal to the number of standard Young tableaux on skew-diagrams. Then, we relate $b$ to certain values of a modified Hermite polynomial.  Third, we relate $b$ to a certain cell decomposition on $\ms{X}$ previously studied by De Concini, Springer, and Strickland.  Using these, we give asymptotic estimates for $b$ as the dimension of the quadrics increases.
\end{abstract}


\section{Introduction}

Let $G$ denote the special linear group $\mt{SL}_n$ over the complex numbers and let $B\subset G$
denote its Borel subgroup of upper triangular matrices.
An important theme in geometric representation theory is the
combinatorial study of the orbits of $B$ when $G$ acts on an algebraic variety $X$.

There are two main sources of examples for which the enumeration becomes as explicit as possible;

1. $X$ is a homogenous space of $G$ of the form $X= G/P$, where $P$ is a parabolic subgroup
containing $B$. Then $P=P_I$ is determined canonically by a finite set $I \subseteq \{1,\dots, n-1\}$, and the
total number $b(X)$ of $B$-orbits in $X$ is given by
\begin{align}\label{A:homogenousenumeration}
b(X)  = \frac{n!}{n_1 ! \cdots n_k !},
\end{align}
where $n_i-1$ ($i=1,\dots, k$) are the lengths of the maximal sequences of consecutive integers in $I$.
For example, if $I=\{2,5,6,7,10,11\}$, then $n_1=2$, $n_2=4$, and $n_3=3$.

2. $X$ is the projectivization of a linear algebraic monoid with unit group $G$.
Then $X$ is a projective $G\times G$-equivariant embedding of $G$ \cite{Rittatore98}, and its $B\times B$-orbits are parametrized
by the nonzero elements of a finite inverse semigroup, called the {\em Renner monoid} \cite{Renner89}.
When the corresponding monoid has a unique minimal $G\times G$-orbit, the order of its Renner monoid is computed in \cite{LLC06}.
In particular, when $X$ is the `equivariant wonderful embedding' $\ms{X}_e$ of $G$, the order of its Renner monoid
is given by
\begin{align}\label{A:wonderful}
1+ b(\ms{X}_e)=  1+\sum_{I\subseteq \{1,\dots, n-1\}} \frac{(n!)^2}{n_I},
\end{align}
where $n_I$ is the product $n_I = n_1 ! \cdots n_k !$ as appeared in (\ref{A:homogenousenumeration}) and
$b(\ms{X}_e)$ is the number of $B\times B$-orbits in $\ms{X}_e$.

In this paper we consider the enumeration in the variety of complete quadrics, which is a
wonderful embedding $\ms{X}:=\ms{X}_n$ of the space of smooth quadric hypersurfaces in $\PP^{n-1}$.
A smooth quadric hypersurface $\mc{Q}$ in $\PP^{n-1}$, the complex projective $n-1$ space, is the vanishing locus of a
quadratic polynomial of the form $x^\top A x$, where $A$ is a symmetric, invertible $n\times n$ matrix
and $x$ is a column vector of variables.
The correspondence $\mc{Q} \leftrightsquigarrow A$ is unique up to scalar multiples of $A$.
There is a transitive action of $\mt{SL}_n$ on smooth quadrics induced from the action on matrices:
\begin{align}\label{A:SLnaction}
g \cdot A = \theta(g) A g^{-1},
\end{align}
where $g \in \mt{SL}_n$ and $\theta$ is the involution $\theta(g) = (g^\top)^{-1}$.
Let $\mt{SO}_n \subset \mt{SL}_n$ denote the orthogonal subgroup consisting of matrices $g\in \mt{SL}_n$ such that
$(g^\top)^{-1}=g$.
The stabilizer of the class of scalar diagonal matrices is the normalizer $\widetilde{\mt{SO}}_n$ of
$\mt{SO}_n$ in $\mt{SL}_n$, hence, the space of smooth quadric hypersurfaces is identified with
$\mt{SL}_n / \widetilde{\mt{SO}}_n$.

Any pair $(G,\sigma)$, where $G$ is a semi-simple, simply connected complex algebraic group
and $\sigma: G\rightarrow G$ an involution has a (canonical) wonderful embedding.
Let $H$ denote the normalizer of $G^\sigma$. The {\em wonderful embedding} $X$ is the unique smooth projective
$G$-variety containing an open $G$-orbit isomorphic to $G/H$ whose boundary $X - (G/H)$ is a union of smooth
$G$-stable divisors with smooth transversal intersections. Boundary divisors are canonically indexed by the elements
of a certain subset $\varDelta$ of a root system associated to $(G,\sigma)$. Each $G$-orbit in $X$ corresponds to a subset
$I\subseteq \varDelta$.  The Zariski closure of the orbit is smooth and is equal to the transverse intersection of the
boundary divisors corresponding to the elements of $I$.

The  wonderful embedding of the pair $(\mt{SL}_n,\sigma)$ above is $\ms{X}$,
where $\sigma(g) = (g^{\top})^{-1}$. In this case, $\varDelta$ is the set of simple roots associated to $\mt{SL}_n$
relative to its maximal torus of diagonal matrices contained in the Borel subgroup $B\subseteq \mt{SL}_n$
of upper triangular matrices, and is canonically identified with the set $[n-1] = \{1, 2, \dots, n-1\}$.

We count the number of orbits of the Borel subgroup $B$ in $\ms{X}$ in three different ways.
Two of these methods are combinatorial in nature and related to each other, while the third method is more geometric.
We compare our findings with the number of $B\times B$-orbits in the equivariant wonderful embedding of $\mt{SL}_n$.
Let us explain our results in more detail.

We identify a partition $\lambda=(\lambda_1 \geq \dots \geq \lambda_k \geq 0)$ with its
Young diagram (left justified rows of boxes with row lengths $\lambda_i$, decreasing from top to bottom).
Given two partitions $\lambda$ and $\mu$, the {\em skew-diagram} $\lambda \times \mu$
is defined by placing $\mu$ below $\lambda$ in such a way that the rightmost box in the first row of $\mu$
is immediately below and to the left of the lowest box in the first column of $\lambda$.
For example, the skew-diagram $(3,2)\times (2,1,1)$ is depicted as in Figure \ref{fig:skew diagram}.
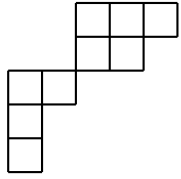
\begin{figure}
\begin{center}
\begin{tikzpicture}[scale=.45]
\begin{scope}
\draw[-, thick] (0,0) to (0,3);
\draw[-, thick] (-2,1) to (2,1);
\draw[-, thick] (0,3) to (3,3);
\draw[-, thick] (0,2) to (3,2);
\draw[-, thick] (2,1) to (2,3);
\draw[-, thick] (1,1) to (1,3);
\draw[-, thick] (3,2) to (3,3);
\draw[-, thick] (-2,1) to (-2,-2);
\draw[-, thick] (0,0) to (-2,0);
\draw[-, thick] (-1,1) to (-1,-2);
\draw[-, thick] (-1,-2) to (-2,-2);
\draw[-, thick] (-1,-1) to (-2,-1);
\end{scope}
\end{tikzpicture}
\end{center}
\label{fig:skew diagram}
\caption{The skew-diagram $(3,2) \times (2,1,1)$.}
\end{figure}
More generally, the skew-diagram $\lambda^1 \times \cdots \times \lambda^r$ of arbitrarily many partitions
$\lambda^1,\lambda^2,\dots,\lambda^r$ is defined, inductively, by the same procedure.

A {\em skew standard Young tableau} (skew SYT for short) is a filling of a skew-diagram with $n$ boxes, using each number
$1,\dots, n$ exactly once, increasing left to right in rows and increasing along columns from top to bottom.
For example,
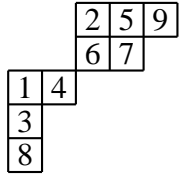
\begin{center}
\begin{tikzpicture}[scale=.45]

\node at (-1.5,.5) {$1$};
\node at (-.5,.5) {$4$};
\node at (-1.5,-.5) {$3$};
\node at (-1.5,-1.5) {$8$};

\node at (.5,2.5) {$2$};
\node at (1.5,2.5) {$5$};
\node at (2.5,2.5) {$9$};
\node at (.5,1.5) {$6$};
\node at (1.5,1.5) {$7$};

\begin{scope}
\draw[-, thick] (0,0) to (0,3);
\draw[-, thick] (-2,1) to (2,1);
\draw[-, thick] (0,3) to (3,3);
\draw[-, thick] (0,2) to (3,2);
\draw[-, thick] (2,1) to (2,3);
\draw[-, thick] (1,1) to (1,3);
\draw[-, thick] (3,2) to (3,3);
\draw[-, thick] (-2,1) to (-2,-2);
\draw[-, thick] (0,0) to (-2,0);
\draw[-, thick] (-1,1) to (-1,-2);
\draw[-, thick] (-1,-2) to (-2,-2);
\draw[-, thick] (-1,-1) to (-2,-1);
\end{scope}
\end{tikzpicture}
\end{center}
is a skew SYT on $(3,2) \times (2,1,1)$.

Our first result is the following
\begin{Theorem} \label{T:Main1}
The number $b(\ms{X})$ of $B$-orbits in $\ms{X}$ is equal to the number of skew standard Young tableaux with $n$ boxes.
\end{Theorem}

Orthogonal polynomials, which are defined as the eigenfunctions of a second order differential operator
of the form
$$
L = p(z) \frac{\ud^2}{ \ud z^2} + q(z) \frac{\ud}{\ud z},
$$
where $p(z)$ and $q(z)$ are (quadratic) polynomials, play an important role in mathematics and physics, \cite{AAR99}.
Among the special cases are the Hermite polynomials (with $p(z)= 1$, $q(z)= -2z$)
and Laguerre polynomials (with $p(z)  = z$, $q(z)= a+1-z$).
Laguerre polynomials arise in the enumeration of $B$-orbits in the most well-known
equivariant embedding of $\mt{SL}_n$, namely the monoid of $n \times n$ matrices \cite{CR08}.

The $B$-orbits in the quasi-affine space $\mt{SL}_n/\mt{SO}_n$ are in bijection with
involutions of the symmetric group $S_n$ \cite{RS90}, and there is a relation between the number of involutions and 
Hermite polynomials through their generating functions. We exploit this connection to relate Hermite polynomials to 
$B$-orbit enumeration in $\ms{X}$. To this end, define the {\em modified-Hermite polynomial}, $H_n(y)$, by
\begin{align}\label{A:hermite}
e^{y(x+x^2/2)} = \sum_{n\geq 0} \frac {H_n(y)}{n!} x^n.
\end{align}

\begin{Theorem}\label{T:second}
The number of $B$-orbits in the variety of complete quadrics $\ms{X}$ is equal to
\begin{align*}
b(\ms{X})= a_{n,1} H_n (1) + a_{n,2} H_n (2) + \cdots + a_{n,n} H_n(n),
\end{align*}
where
\begin{align*}
a_{n,r} = \sum_{i=0}^{n-r} (-1)^i {r+i \choose r}.
\end{align*}
\end{Theorem}

Let $I \subset [n-1]$ and let $P =P_I := B W_I B$ denote the parabolic subgroup corresponding to $W_I$,
the subgroup of $W := S_n$ generated by the simple transpositions $s_i = (i, i+1)$ for $i \in I$.
Let $W^I$ denote the minimal length coset representatives of $W / W_I$.
The Bruhat-Chevalley decomposition of $\mt{SL}_n$ into the disjoint union of $B w P$, $w \in W^I$,
gives a cell decomposition for the partial flag variety $X=\mt{SL}_n/P$.
The value at 1 of the Poincar{\' e} polynomial of $X$ is the number of $B$-orbits in $X$.

In general, a cell decomposition need not coincide with the decomposition of the variety into $B$-obits.
On the other hand, it is shown in \cite{BL87} that the cells of a known cell decomposition of a wonderful embedding
are unions of $B$-orbits.
The structure of these cell decompositions for a general wonderful embedding is determined by
De Concini and Springer in \cite{DS85} and the case of complete quadrics is given by Strickland in \cite{Strickland86}.

Building on these geometric observations we obtain our third formula for $b(\ms{X})$, expressible
in terms of a classical notion of combinatorics, namely the descent set of a permutation:
if $w=w_1w_2\cdots w_n \in S_n$ is a permutation given in one-line notation, then the {\em descent set} of
$w$ is defined as $\Des(w)= \{i :\ w_i > w_{i+1}\}$.

\begin{Theorem}\label{T:third}
The number of $B$-orbits of complete quadrics $\ms{X}$ is equal to
\begin{align*}
b(\ms{X}) = \sum_{J} \sum_{w \in W^J} 2^{a_J(w)+b_J(w)},
\end{align*}
where the first summation is over those subsets $J$ of $[n-1]= \{1,2,\dots, n-1\}$ which do not contain two consecutive numbers, $a_J(w)$ is the number of descents $i$ of $w$ such that neither $i$ nor $i-1$ is in $J$, and $b_J(w)$ is the number of descents $i$ of $w$ such that $w(i+1) < w(i-1) < w(i)$, $i-1 \in J$ and $i \notin J$.
\end{Theorem}

\begin{Remark}
Theorem \ref{T:third} suggests studying the following family of $q$-polynomials:
\begin{align*}
B_{n,J}(q) = \sum_{w \in W^J} q^{a_J(w)+b_J(w)},
\end{align*}
where $J$ is {\em any} subset of $[n-1]$, $a_J(w)$ and $b_J(w)$ are as in the theorem.
The following observations about $B_{n,J}(q)$ are easily verified.
\begin{enumerate}
\item When $J=\emptyset$, $a_J(w)$ is equal to $\Des(w)$ and $b_J(w)=0$.
Therefore, $B_{n,\emptyset}(q)$ is the Eulerian polynomial $\sum_{w\in W} q^{|\Des(w)|}$.
\item On the other extreme, if $J=[n-1]$, then $B_{n,J}(q)=1$.
\item For $J\subseteq [n-1]$, let $m_1,\dots,m_s$ be the sizes of the maximal sequences of consecutive elements in $J$. Then
$$
B_{n,J}(1) = \frac{n!}{(m_1+1) ! \cdots (m_s+1) !}.
$$
\end{enumerate}
At the end of the manuscript, in Table \ref{Table}, we have the list of all $B_{n,J}$ for $n=5$.
We conjecture that, for each $J\subseteq [n-1]$, the polynomial $B_{n,J}(q)$ is unimodal.
\end{Remark}

It is natural to compare $b(\ms{X})$ with $b(\ms{X}_e)$, the number of $B\times B$-orbits in the equivariant wonderful embedding
of $\mt{SL}_n$.
To this end, let $b_n$ denote the $n$-th {\em ordered Bell number}, the number of ordered set partitions of
an $n$-element set. For example, $b_2 = 3$, since the ordered partitions of $\{1,2\}$ are $(\{ 1\}, \{ 2\})$, $(\{2\}, \{1\})$, and $(\{1,2\}).$
We know from \cite[Section 5.2]{Wilf06} that, for large $n$,
\begin{align}\label{A:asymptoticBell}
b_n \thicksim \frac{n!}{2 (\log 2)^{n+1}} \thicksim \frac{n! (1.443)^{n-1}}{0.961}.
\end{align}

The next lemma, whose proof is given in Section \ref{S:upper and lower}, is a hint
of the rich combinatorial structure on the set of $B\times B$-orbits in $\ms{X}_e$.
\begin{Lemma}\label{L:thelemma}
The number $b(\ms{X}_e)$ of $B\times B$-orbits in the equivariant wonderful embedding of $\mt{SL}_n$ is
equal to $n! b_n$.
\end{Lemma}

Let $F_n$ denote the $n$-th Fibonacci number. For large $n$,
\begin{align}\label{A:asymptoticFibonacci}
F_n \thicksim \frac{1}{\sqrt{5}} \left( \frac{1+ \sqrt{5}}{2} \right)^n \thicksim \frac{(1.618)^{n-1}}{1.382}.
\end{align}

\begin{Theorem}\label{T:final}
For all sufficiently large $n$, the following inequalities hold:
$$
F_n n! < b(\ms{X}) < 2^{n-1} n! < b(\ms{X}_e).
$$
\end{Theorem}

\noindent \textbf{Acknowledgements.}
We thank Armin Straub for the communication on the details of Lemma \ref{L:thelemma},
and thank Tewodros Amdeberhan for his suggestions which improved the quality of the paper.
We thank Zhenheng Li, Victor Moll, and Michelle Wachs for helpful conversations.
The first author is partially supported by the Louisiana Board of Regents enhancement grant.

We thank the anonymous referees for providing many thoughtful suggestions to improve the quality of our paper.

\section{Notation and Preliminaries}\label{S:notation}

\subsection{Notation and Conventions}
All varieties are defined over $\C$ and all algebraic groups are complex algebraic groups.
Throughout, $n$ is a fixed integer, and $\ms{X} := \ms{X}_n$ denotes the $\mt{SL}_n$-variety of $(n-2)$-dimensional
complete quadrics, which is reviewed in Section \ref{S:complete quadrics}.
The integer interval $[m]$ denotes $\{1, 2, \dots, m\}$, the complement of $I$ in $[m]$ is denoted by $I^\text{c}$.
If $I$ and $K$ are sets, then $I-K$ denotes the set complement $\{a\in I:\ a\notin K\}$.
The transpose of a matrix $A$ is denoted $A^\top$.

The symmetric group $S_n$ of permutations of $[n]$ is denoted by $W$, and for $w \in W$,
$\ell(w)$ denotes the {\em length} of $w$, $\ell(w)= |\{(i,j):\ 1\leq i <j \leq n,\ w(i) > w(j)\}|$.

Let $B$ be the Borel subgroup of upper triangular matrices in $G=\mt{SL}_n$. Let $T\subset B$
denote the maximal torus of diagonal matrices.
Let $\vi_1,\dots, \vi_n$ denote the standard basis vectors for the Euclidean space $\R^n$ and
let $\alpha_i = \vi_i - \vi_{i+1}$ for $i=1,\dots,n-1$.
The standard root theory notation we use is as follows:
\begin{eqnarray*}
 \varDelta &=& \{ \alpha_i :\  1\leq i\leq n-1\}\ \text{(the set of simple roots relative to}\ (\mt{SL}_n,T)),\\
 \varPhi^+ &=& \{ \vi_i- \vi_j:\  1\leq i<j \leq n \}\ \text{(the set of positive roots associated with}\ \varDelta),\\
 \varPhi &=& \varPhi^+ \cup -\varPhi^+ \ \text{(the set of all roots associated with}\ \varDelta).
\end{eqnarray*}

The symmetric group $W$ acts on $\varPhi$ by permuting the indices
of standard basis vectors: $w\cdot \vi_i = \vi_{w^{-1}(i)}$.
 For $J\subseteq \varDelta$, let $W_J$ denote the
{\em parabolic subgroup} generated by the simple transpositions $\sigma_i=(i,i+1)$ corresponding to the roots $\alpha_i$ in $J$.
Let $W^J$ denote the {\em minimal length coset representatives} of $W/W_J$.
Since $W_J$ is a product of symmetric subgroups of $W$ generated by the adjacent simple roots contained in $J$,
the number of elements of $W^J$ is $n!/ n_1 !\cdots n_k !$, where $n_i-1$, $i=1,\dots, k$ is the length of a maximal
sequence of consecutive integers from $J$.

An involution $w\in W$ is an element of order $\leq 2$. We know from \cite{RS90} that involutions are in one-to-one correspondence with
$B$-orbits in the space of invertible symmetric $n\times n$ matrices. The number of involutions in $W$ is denoted by $I(n)$.
Then, as in [Section 3.8, \cite{Wilf06}],
\begin{align}\label{A:involution series}
Q(x) := \sum_{n\geq 0} \frac{I(n)}{n!} x^n = e^{x+x^2/2}.
\end{align}

We preserve our notation on partitions from Introduction. In addition, if $\lambda^1,\dots, \lambda^r$ are
partitions, then $|\lambda^1 \times \cdots \times  \lambda^r|$ denotes the sum of entries of $\lambda^i$'s for $i=1,\dots, r$.
For a partition $\lambda$, let $f^\lambda$ denote the number of $SYT$ of shape $\lambda$.
It follows from the well known RSK algorithm (see \cite{Sagan}, for example) that the number of involutions
is equal to the total number of $SYT$:
\begin{equation}\label{A:I=sumflambda}
I (n) = \sum_{|\lambda|= n} f^\lambda.
\end{equation}
The definition of a $SYT$ is extended to skew-shapes in the obvious way.
The number of $SYT$ of shape $\lambda \times \mu$ is equal to $f^{\lambda \times \mu}= {n+m \choose n} f^\lambda f^\mu$.
More generally, for given partitions $\lambda^1,\lambda^2,\dots, \lambda^k$
such that $ |\lambda^1 \times \cdots \times \lambda^k | =n$, we have
\begin{align}\label{A:fcross}
f^{\lambda^1 \times \dots \times \lambda^k} = \frac{n!}{ |\lambda^1| ! |\lambda^2| ! \cdots |\lambda^k|!} \prod_{i=1}^k f^{\lambda^i}
= {n \choose |\lambda^1|, |\lambda^2|,\dots, |\lambda^k|} \prod_{i=1}^k f^{\lambda^i}.
\end{align}

A {\em composition} of $n$ is an ordered sequence $\gamma = (\gamma_1,\dots, \gamma_k)$ of positive integers such that
$\sum \gamma_i  = n$. We denote by $Comp(n,k)$ the set of all compositions of $n$ with $k$ nonzero parts, and
denote by $Comp(n)$ the set of all compositions of $n$.

There exists a bijection between the set of all compositions of $n$ and the set of all subsets
of $[n-1]$. For $\gamma = (\gamma_1,\dots, \gamma_k)$ a composition of $n$, let $I_\gamma$
denote the complement of the set $I=\{ \gamma_1,\gamma_1+\gamma_2,\dots, \gamma_1+\cdots+\gamma_{k-1}\}$
in $[n-1]$. Then
\begin{align}\label{A:bijectionfromcompositions}
\gamma \mapsto I_\gamma
\end{align}
is the desired bijection.

\subsection{Partial Flag Varieties}\label{S:partialflagvariety}

We briefly review the relevant terminology from the theory of homogenous spaces, referring the reader to \cite{Humphreys}
for more on details.

Let $G$ denote $\mt{SL}_n$, $B\subset G$ the subgroup of upper triangular matrices and let $T\subset B$ be the subgroup
diagonal matrices.
A closed subgroup $P$ of $G$ is called a {\em standard parabolic subgroup}, if $B\subseteq P$.
More generally, a subgroup $P'$ is called {\em parabolic}, if there exists $g\in G$ such that $P'=g Pg^{-1}$
for some standard parabolic subgroup $P\subset G$.

There exists a one-to-one correspondence between the standard parabolic subgroups of $G$ and the subsets
of $[n-1]$ given by:
$$
I\rightsquigarrow P_I:= B W_I B.
$$
With this notation, the quotient space $G/P_I$ is called the {\em partial flag variety} of type $I$.
We need the following interpretation of $G/P_I$.

Let $I^c=\{j_1,\dots, j_{s}\}$ be the complement of $I$ in $[n-1]$.
Then $G/P_I$ is identified with the set of all nested sequences of vector spaces (flags) of the form
$$
\mc{F}:\ 0 \subset V_{j_1} \subset V_{j_2} \subset \cdots \subset V_{j_s} \subset \C^n,
$$
where $\dim V_{j_k} = j_k$, $k=1,\dots, s$.

Finally, let us mention the fact that $G/P_I$ has a canonical decomposition into $B$-orbits where the orbits are indexed by
the elements of $W^I$, the minimal length coset representatives.
Furthermore, $T$ acts on $G/P_I$ and each $B$-orbit contains a unique $T$-fixed flag.

\subsection{Wonderful Embeddings}\label{S:wonderful embeddings}
		
We briefly review the theory of wonderful embeddings, referring the reader to \cite{DP83} for more details.

A symmetric space is a quotient of the form $G/H$, where $G$ is an algebraic group, $H$ is the normalizer of
the fixed subgroup of an automorphism $\sigma: G \rightarrow G$ of order 2.
When $G$ is a semi-simple simply connected algebraic group (over an algebraically closed field), there exists a unique minimal
smooth projective $G$-variety $X$, called the {\em wonderful embedding} of $G/H$, such that
\begin{enumerate}
\item $X$ contains an open $G$-orbit $X_0$ isomorphic to $G/H$
\item $X - X_0$ is the union of finitely many $G$-stable smooth codimension one subvarieties $X_i$ for $i = 1, 2, \dots, r$
\item for any $I \subset [r]$, the intersection $X^I := \bigcap_{i \notin I} X_i$ is smooth and transverse
\item every irreducible $G$-stable subvariety has the form $X^I$ for some $I \subset [r]$
\item for $I\subset [r]$, let $\ms{O}^I$ denote the corresponding $G$-orbit whose Zariski closure is $X^I$.
There is a fundamental decomposition
\begin{equation}\label{E:orbit union}
X^I = \bigsqcup_{K \subset I} \ms{O}^K.
\end{equation}
Consequently, $G$-orbit closures form a Boolean lattice and there exists a unique closed orbit $Z$,
corresponding to $I = \emptyset$.

\end{enumerate}

The recursive structure of wonderful embeddings is apparent from the following observation:

For each $I\subset [r]$, there exists a parabolic subgroup $P_I$ and $G$-equivariant fibrations $\pi: \ms{O}^I \rightarrow G/P_I$
and $\overline{\pi}_I : X^I \rightarrow G/P_I$ such that the following diagram commutes:

\begin{figure}[htp]
\centering
\begin{tikzpicture}[scale=.45]
\begin{scope}
\node at (-4,0) (a) {$\ms{O}^I$};
\node at (4,0) (b) {$X^I$};
\node at (0,-4.5) (c) {$G/P_I$};
\node at (-2.5,-2.5) (d) {$\pi_I$};
\node at (2.6,-2.5) (d) {$\overline{\pi}_I$};
\end{scope}
\begin{scope}
\draw[right hook->, thick] (a) to (b);
\draw[->, thick] (a) to (c);
\draw[->, thick] (b) to (c);
\end{scope}
\end{tikzpicture}
\end{figure}
\noindent Furthermore, the fiber $(\overline{\pi}_I)^{-1}(x)$ over a point $x \in G/P_I$ is isomorphic to
a wonderful embedding of $\pi_I^{-1}(x)$.

\subsection{Complete Quadrics}\label{S:complete quadrics}

There is a vast literature on the variety $\ms{X}$ of complete quadrics.  See \cite{Laksov87} for a survey.
We briefly recall the relevant definitions.

Let $\text{Sym}_n$ denote the space of invertible symmetric $n\times n$ matrices and let $X_0$ denote the
projectivization of $\mt{Sym}_n$.
Thus, $X_0$ is identified with the symmetric space $\mt{SL}_n/\widetilde{\mt{SO}}_n$, or, equivalently,
with the space of smooth quadric hypersurfaces in $\PP^{n-1}$.

The classical definition of $\ms{X}$ (see \cite{Schubert, Semple48, Tyrrell56}) is as the closure of the image of the map
$$
[A] \mapsto ([\Lambda^1(A)], [\Lambda^2(A)], \dots, [\Lambda^{n-1}(A)]) \in \prod_{i=1}^{n-1} \PP(\Lambda^i(\text{Sym}_n)).
$$

The modern definition of $\ms{X}$ is as the wonderful embedding of $X_0\cong \mt{SL}_n/\widetilde{\mt{SO}}_n$.
Let $I\subseteq [n-1]$ and let $P_I$ be the corresponding standard parabolic subgroup.
It follows from the results of Vainsencher in \cite{Vainsencher82} that a point $\mathcal{P} \in \ms{X}$ is described by the data of a flag
\begin{equation}\label{E:flag}
\mathcal{F}: V_0 = 0 \subset V_1 \subset \dots \subset V_{s-1} \subset V_s =\C^n
\end{equation}
and a collection $\mathcal{Q} = (Q_1, \dots Q_s)$ of quadrics, where $Q_i$ is a quadric in $\PP(V_i)$ whose singular locus
is $\PP(V_{i-1})$. Moreover, $\ms{O}^I$ consists of complete quadrics whose flag $\mathcal{F}$ is of type $I$
and the map $(\mathcal{F}, \mathcal{Q}) \mapsto \mathcal{F}$ is the $\mt{SL}_n$-equivariant projection
\begin{equation}\label{E:pi_K}
\overline{\pi}_I : \ms{X}^I \rightarrow \mt{SL}_n / P_I.
\end{equation}
The fiber of $\overline{\pi}_I$ over $\mathcal{F} \in \mt{SL}_n / P_I$ is isomorphic to a product of varieties of complete quadrics of
smaller dimension.

\section{Proofs}\label{S:proofs}

To parameterize the $B$-orbits on the space of complete quadrics, first observe that any complete quadric can be transformed
by the action of a suitable element of $B$ to a complete quadric whose associated flag is torus-fixed.
Thus, the study of $B$-orbits of complete quadrics reduces to two separate problems:  the (easy) enumeration of torus-fixed partial
flags and, for each such flag, the enumeration of orbits which contain a complete quadric whose defining flag is the given flag.

We are ready to prove Theorem \ref{T:Main1}, namely
\begin{align}\label{A:main}
b(\ms{X}) = \sum_{ \lambda^1 \times \cdots \times \lambda^k \in C(n)}  f^{\lambda^1 \times \cdots \times \lambda^k},
\end{align}
where
$C(n) = \{ \lambda^1 \times \cdots \times \lambda^k :\ \lambda^i \ \text{is a partition and}\
|\lambda^1 \times \cdots \times \lambda^k| = n \}$.

\begin{proof}[Proof of Theorem \ref{T:Main1}]

Let $V$ denote an ambient vector space of dimension $n$. Recall that in the case of the trivial flag, $\mathcal{F}: 0 \subset V$,
the $B$-orbits of non-degenerate quadrics are in one-to-one correspondence with symmetric $n\times n$ permutation matrices
or, equivalently, the involutions of $W$.
	
In general, suppose that our flag has the form
$$
\mathcal{F}: 0 \subset V_{i_1} \subset V_{i_1 + i_2} \subset \dots \subset
V_{i_1 + i_2 + \cdots + i_{k-1}} \subset V_{i_1 + i_2 + \cdots + i_k} = V
$$
where $i_1 + i_2 + \cdots + i_k = n$ is a composition of $n$. Here, each $V_j$ is a vector space of dimension $j$.
To give a torus-fixed flag of this form is equivalent to choosing $i_1$ standard basis vectors to span $V_{i_1}$, $i_2$
more standard basis vectors to span $V_{i_1+ i_2}$, and so on. It follows that there are ${n \choose i_1, i_2, \dots, i_k}$
torus-fixed flags of this type.
	
Given such a torus-fixed flag, to specify a complete quadric is equivalent to specify a non-degenerate quadric on each successive
quotient space
$$
V_{i_1 + i_2 + \cdots + i_j} / V_{i_1 + i_2 + \cdots + i_{j-1}},\ 1 \leq j \leq k.
$$
Let $I\subseteq [n-1]$ denote the type of a partial flag $\mc{F}$ and let $B_{\mc{F}} \subset B$ denote the stabilizer of
$\mathcal{F}$, viewed as a point in the appropriate partial flag variety.
When the subgroup $B_\mathcal{F}$ is restricted to one of the successive quotient spaces, it acts as the full Borel subgroup of
upper-triangular matrices.  Therefore, the total number of $B$-orbits of quadrics of this type is
equal to $I(i_1) I(i_2) \cdots I(i_k)$.

It follows from these observations that the number of $B$-orbits in the space of complete quadrics is
\begin{align}\label{A:compositions}
b(\ms{X}) = \sum_{ \stackrel{ \gamma \in Comp(n) }{\gamma= (i_1,\dots,i_k)} } {n \choose i_1, i_2, \dots,i_k}
I(i_1) I(i_2) \cdots I(i_k),
\end{align}
where the summation is over all compositions of $n$. Combined with (\ref{A:I=sumflambda}) and  (\ref{A:fcross}),
(\ref{A:compositions}) gives the desired equality (\ref{A:main}).

\end{proof}

\subsection{Modified-Hermite polynomials}

Let us define {\em modified-Hermite polynomials}, $H_k(y)$, $k=0,1,\dots$ by
\begin{align}\label{A:hermite}
H(x;y):= e^{y(x+x^2/2)} = \sum_{k\geq0 } \frac {H_k(y)}{k!} x^k.
\end{align}
Differentiating with respect to $x$ and rewriting the generating series give
\begin{align}\label{A:recurrence}
H_{k+1}(y) = y(H_k(y) + k H_{k-1}(y)),\ k\geq 1.
\end{align}
The data in the following table is easily verified:
\begin{center}
\begin{tabular}{c|c}
$k$ & $H_k(y)$ \\
\hline
0 & 1 \\
1 & y \\
2 & $y^2 + y$ \\
3 & $y^3 + 3y^2$ \\
4 & $y^4 + 6y^3 + 3 y^2$
\end{tabular}
\end{center}

We are ready to prove our second theorem, which states
\begin{align*}
b(\ms{X})= a_{n,1} H_n (1) + a_{n,2} H_n (2) + \cdots + a_{n,n} H_n(n),
\end{align*}
where ${\ds a_{n,r} = \sum_{i=0}^{n-r} (-1)^i {r+i \choose r}}$.

\begin{proof}[Proof of Theorem \ref{T:second}]

Let $Q(x)$ be, as in (\ref{A:involution series}), the exponential generating series for the number of involutions. Clearly, $H(x,y)= Q(x)^y$.
Expanding $Q(x)^k$, we see that the coefficient $H_n(k)$ of $x^n/n!$ in $Q(x)^k$ is equal to
\begin{align} \label{A:Borels1}
H_n(k) = \sum_{ j_1 + j_2 + \cdots + j_k =n } {n \choose j_1, j_2, \dots, j_k}
I(j_1) I(j_2) \cdots I(j_k),
\end{align}
where the summation is over all $k$-sequences of non-negative integers $j_i$ summing to $n$ ([\cite{Wilf06}, Section 2.3]).
Note that $j_i$, for some $i\in [k]$, is allowed to be equal to 0.

Our goal is to express the right hand side of (\ref{A:compositions}) in terms of $H_n(k)$'s.
To this end, we compute the coefficient of $x^n/n!$ in $\ds{\sum_{k=1}^n \prod_{r=1}^k ( Q(x) -1)}$, which
is equal to the right hand side of (\ref{A:compositions}).
By using Binomial Theorem, we obtain
\begin{align}\label{A:b_Hermite}
b(\ms{X}) = \sum_{k=1}^n \sum_{j=1}^k (-1)^{k-j} {k \choose j} H_n(j).
\end{align}
We rearrange the summation (\ref{A:b_Hermite}) by collecting $H_n(j)$'s for $j=1,\dots, n$,
and by suitably changing the indices of the new summations. This gives the desired formula that
$$
b(\ms{X}) = \sum_{r=1}^n \sum_{i=0}^{n-r} (-1)^{i} {r+i \choose r} H_n(r).
$$
\end{proof}

\begin{Example}
The number of $B$-orbits in $\ms{X}_3$, the variety of complete conics in $\PP^2$ is
\begin{align*}
\sum_{k=1}^3 \sum_{i=0}^{3-k} (-1)^i {k+i \choose k} H_3(k) =
 (1-2+6)4 + (1-3)16 +54 = 22,
\end{align*}
which is confirmed by the Hasse diagram in Figure \ref{fig:Cells} below.
\end{Example}

\subsection{Geometric Counting}

In this section we give a (geometric) proof of Theorem \ref{T:third}, which states that:

The number of $B$-orbits in complete quadrics $\ms{X}$ is equal to
\begin{align*}
b(\ms{X}) = \sum_{J} \sum_{w \in W^J} 2^{a_J(w)+b_J(w)},
\end{align*}
where the first summation is over those subsets $J$ of $[n-1]= \{1,2,\dots, n-1\}$ which does not contain two consecutive numbers, $a_J(w)$ is the number of descents $i$ of $w$ such that neither $i$ nor $i-1$ is in $J$, and $b_J(w)$ is the number of descents $i$ of $w$ such that $w(i+1) < w(i-1) < w(i)$, $i-1 \in J$ and $i \notin J$.

The proof follows from the following Lemma:
\begin{Lemma}\label{L:third}
For $J\subseteq \varDelta$, let $w_J$ denote the longest element in the subgroup $W_J$. Then
\begin{align*}
b(\ms{X}) = \sum_{J} \sum_{w \in W^J} 2^{r_J(w)},
\end{align*}
where the first summation is over those subsets $J\subset \varDelta =\{\alpha_1,\dots,\alpha_{n-1}\}$ satisfying
$(\alpha_i,\alpha_j )=0$ for all $\alpha_i \neq \alpha_j \in J$, and $r_J(w)$ is the number of $\alpha_i$ not in $J$ such that $-w(\alpha_i + w_J(\alpha_i)) > 0$.
\end{Lemma}

Before we start the proof of the lemma, we recall a fundamentally important result of Bia{\l}ynicki-Birula
about cell decompositions of varieties.
Let $X$ be a smooth projective variety over $\C$ on which an algebraic torus $T$ acts with finitely many fixed points.
Let $T'$ be a one-parameter subgroup, whose set of fixed points $X^{T'}$ is equal to that $X^T$ of $T$.
For $p\in X^{T'}$ define the sets $C_p^+ = \{y \in X:\ \ds{\lim_{t \to 0} t \cdot y = p,}\ t \in T' \}$ and
$C_p^- = \{y \in X:\ \ds{\lim_{t \to \infty} t \cdot y = p,}\ t \in T' \}$,
called the {\em plus cell} and {\em minus cell} of $p$, respectively.

\begin{Theorem}[\cite{BB73}]\label{T:BB}
Let $X,T$ and $T'$ be as above. Then
\begin{enumerate}
\item $C_p^+$ and $C_p^-$ are locally closed subvarieties isomorphic to an affine space;
\item if $T_p X$ is the tangent space of $X$ at $p$, then $C_p^+$ (resp., $C_p^-$) is $T'$-equivariantly
isomorphic to the subspace $T_p^+ X$ (resp., $T_p^- X$) of $T_p X$ spanned by the positive (resp., negative)
weight spaces of the action of $T'$ on $T_p X$.
\end{enumerate}
\end{Theorem}

As a consequence of Theorem \ref{T:BB}, there exists a filtration
$$
X^{T'}  = V_0 \subset V_1 \subset \cdots \subset V_n = X,\qquad n = \dim X,
$$
of closed subsets such that for each $i=1,\dots,n$, $V_i - V_{i-1}$ is the disjoint union of
the plus (resp.,  minus) cells in $X$ of (complex) dimension $i$.

\begin{proof}[Proof of Lemma \ref{L:third}]

It is shown by De Concini and Procesi \cite{DP83} that a wonderful embedding $X$ of $G/H$ has a cell decomposition
arising from the action of a suitable maximal torus $T$ of $G$ on $X$.
In \cite{BL87} Brion and Luna show that if a cell of the embedding intersects a $G$-orbit nontrivially, then the intersection is a single
Borel orbit.
Therefore, to count the number of $B$-orbits in $X$ it is enough to determine 1) the number of cells, equivalently,
the number of torus fixed points of $X$; 2) the number of $G$-orbits intersecting a given cell.

For the first item, we refer to a result of Strickland:
\begin{Lemma}[\cite{Strickland86}, Proposition 2.1]\label{L:Strickland}
An $\mt{SL}_n$-orbit $\ms{O}^K \subset \ms{X}$, $K\subseteq \varDelta$ contains a torus fixed point if and only if
$(\alpha_i,\alpha_j)=0$ for all $\alpha_i \neq \alpha_j \in K$.
Furthermore, in this case, there exists a a bijection between the torus fixed points contained in $\ms{O}^K$
and the set of minimal coset representatives $W^K$.
\end{Lemma}

A subset of simple roots $K\subset \{\alpha_1,\dots, \alpha_{n-1}\}$ is called {\em special}, if
$K$ is as in Lemma \ref{L:Strickland}: $(\alpha_i,\alpha_j)=0$ for all $\alpha_i \neq \alpha_j \in K$.

For the second item, we need a result of De Concini and Springer [\cite{DS85}, Lemma 4.1] combined with
[\cite{Strickland86}, Proposition 2.5]:
\begin{Lemma}\label{L:DS}
Let $\ms{O}^J$, $J\subset \varDelta$ be an $\mt{SL}_n$-orbit containing a torus fixed point $p\in \ms{O}^J$.
Let $C_p^+$ denote the plus cell centered at $p$. Let $w\in W^J$ denote the minimal length coset representative corresponding
to $p$ and $w_J$ denote the longest element of $W_J$. Then, an orbit $\ms{O}^K$, $K\subset \varDelta$
intersects $C_p^+$ non-trivially if and only if
$J \subset K \subset J \cup \{ \alpha_i \in J^c:\  -w( \alpha_i + w_J(\alpha_i)) > 0\}$.
\end{Lemma}

It follows from Lemma \ref{L:Strickland} and Lemma \ref{L:DS} that to count the number of $B$-orbits in $\ms{X}$
it is enough to count the number of subsets of $\{ \alpha_i \in J^c:\  -w( \alpha_i + w_J(\alpha_i)) > 0\}$ for
each $w\in W^J$, $J$ special.
Therefore,
\begin{align*}
b(\ms{X}) = \sum_{J} \sum_{w \in W^J} 2^{r_J(w)},
\end{align*}
where $r_J(w)$ is the number of $\alpha_i$ not in $J$ such that $-w(\alpha_i + w_J(\alpha_i)) > 0$
and the first summation is over all special subsets $J\subset \varDelta$.

\end{proof}

\begin{proof}[Proof of Theorem \ref{T:third}]

Let $J\subset \varDelta$ be a special subset, and let $w_J$ be the longest element in $W_J$.
Then $w_J = \prod_{\alpha_i \in J} \sigma_i$, where $\sigma_i$ is the simple transposition
corresponding to $\alpha_i $.
Let $\alpha_i \notin J$. Then
$$
-w(\alpha_i+w_J (\alpha_i) )=
\begin{cases}
-2\vi_{w(i)}+2\vi_{w(i+1)} & \text{if}\ \{\alpha_{i-1},\alpha_{i+1}\} \cap J = \emptyset,\\
-\vi_{w(i-1)}-\vi_{w(i)}+2\vi_{w(i+1)} & \text{if}\ \{\alpha_{i-1},\alpha_{i+1}\} \cap J = \{ \alpha_{i-1}\},\\
-2\vi_{w(i)}+\vi_{w(i+1)}+\vi_{w(i+2)} & \text{if}\ \{\alpha_{i-1},\alpha_{i+1}\} \cap J = \{ \alpha_{i+1}\},\\
-\vi_{w(i-1)}-\vi_{w(i)}+\vi_{w(i+1)}+\vi_{w(i+2)} & \text{if}\ \{\alpha_{i-1},\alpha_{i+1}\} \cap J = \{ \alpha_{i-1},\alpha_{i+1}\}.
\end{cases}
$$
Arguing as in [Proposition 2.6, \cite{Strickland86}], we obtain $-w(\alpha_i+w_J (\alpha_i) )>0$ if and only if
either $\alpha_{i-1} \notin J$ and $w(i+1) < w(i)$, or $\alpha_{i-1} \in J$ and $w(i+1) < w(i-1) < w(i)$.
It follows from Lemma \ref{L:third} that $r_J(w) = a_J(w) + b_J(w)$,
and the proof is complete.
		
\end{proof}

\begin{Remark}
It follows from Lemma \ref{L:DS} that the inclusion poset formed by the $B$-orbits
within each plus-cell is a Boolean lattice.
\end{Remark}

In Figure \ref{fig:Cells} (at the end of the paper) we depict the cell decomposition of $\ms{X}_3$, the variety of complete conics in $\PP^2$.
Each colored disk represents a $B$-orbit and {\em edges} stand for the covering relations between closures of $B$-orbits.
A {\em cell} is a union of all $B$-orbits of the same color.
We include the label $I\subseteq \{1,2\}$, which indicates the $\mt{SL}_3$-orbit containing the given
$B$-orbit. We use the label $T$ to indicate the presence of a fixed point under the maximal torus $T$ of $\mt{SL}_3$.

\subsection{Upper and Lower Bounds}\label{S:upper and lower}

For a positive integer $n\in \Z$, define $\psi(n) = I(n)/n!$, the ratio of number of involutions to the number of permutations.
Then $\psi(n) \leq 1$.

It follows from (\ref{A:compositions}) that
\begin{align}\label{A:compositionsMuir}
b(\ms{X}) = n! \sum_{
\stackrel{ \gamma \in Comp(n) }{\gamma= (\gamma_1,\dots, \gamma_k)}} \prod_{i=1}^k \psi(\gamma_i).
\end{align}

We are ready to prove Theorem \ref{T:final}, which states:

For all sufficiently large $n$, the following inequalities hold: $b(\ms{X}_e) = b_n n! <F_n n! < b(\ms{X}) < 2^{n-1} n!$.

\begin{proof}
Since each term of the summation (\ref{A:compositionsMuir}) is at most 1, and due to the bijection (\ref{A:bijectionfromcompositions}),
the right hand side of (\ref{A:compositionsMuir}) is bounded by
\begin{align}
n! \sum_{\gamma \in Comp(n) } 1 = n! 2^{n-1}.
\end{align}
This gives the desired upper bound for $b(\ms{X})$.

For the lower bound we use Lemma \ref{L:Strickland}: an $\mt{SL}_n$-orbit
$\ms{O}^K \subset \ms{X}$, $K\subseteq \varDelta$ contains a torus fixed point if and only if
$(\gamma_i,\gamma_j)=0$ for all $\gamma_i \neq \gamma_j \in K$.
Equivalently, indices of the elements from $K$ are non-consecutive. The total number of such
$K \subset \{1,\dots, n-1\}$ is given by the $n$-th Fibonacci number, viewing the elements of
$K$ as colorings of the corresponding nodes of the Coxeter graph of $\varDelta$. Testing the first node being colored or not gives
the recurrence of the Fibonacci numbers.

It follows from \cite{Brion98}, Section 1.2 that the number of $B$-orbits in $\ms{O}^K$, for $K$ special is $n!$.
Combining these observations we see that the total number of $B$-orbits in $X$ is larger than $n! F_n$.
Hence we obtain the inequalities:
$$
F_n n! < b(\ms{X}) < 2^{n-1} n! \ (\text{for all}\ n\geq 3).
$$
It is easy to check that the inequalities are in fact equalities for $n=1$ and $2$.

Next we prove Lemma \ref{L:thelemma}, which states that the number $b(\ms{X}_e)$ of $B\times B$-orbits in the minimal equivariant embedding $\ms{X}_e$ is $n!$ times
the number of ordered set partitions $b_n$.

\begin{proof}
We know from [Section 5.2, \cite{Wilf06}] that the generating series $\ds{\sum_{n\geq 0} \frac{b_n}{n!}x^n }$ is
$\ds{\frac{1}{2-e^x}}$. For $n\in \Z$ positive, let $M(n)$ denote the sum
$$
M(n)= \sum_{k\geq0} \sum_{\stackrel{i_1+\cdots + i_k=n}{i_j \geq 1}} { n! \choose i_1,i_2,\dots,i_k},
$$
and set $M(0)=1$.
Since
$$
\ds{
\sum_{n\geq 1} \sum_{\stackrel{i_1+\cdots + i_k=n}{i_j \geq 1}} { n! \choose i_1,i_2,\dots,i_k} \frac{x^n}{n!}=
\sum_{n\geq 1} \sum_{\stackrel{i_1+\cdots + i_k=n}{i_j \geq 1}}  \frac{x^n}{i_1! i_2! \cdots i_k!} = \left( \sum_{n\geq 1} \frac{x^n}{n!} \right)^k
} = (e^x-1)^k,$$
the generating series $\ds{\sum_{n\geq 0}\frac{M(n)}{n!} x^n}$ is equal to $\ds{\frac{1}{1-(e^x-1)}=\frac{1}{2-e^x}}.$
Therefore, $M(n)= b_n$ for $n\geq 0$ and it remains to show the equality
$n! M(n) = \ds{\sum_{I \subseteq [n-1]} \frac{(n!)^2}{ n_I}}$, where $n_I = n_1 ! n_2 ! \cdots n_k!$ and
$n_i$ is 1 more than the length of a longest sequence of consecutive integers that appear in $I$.

Let $\gamma = \gamma_1+\cdots + \gamma_k$ be a composition of $n$ and let $I_\gamma$ be the associated
subset under the bijection (\ref{A:bijectionfromcompositions}), the complement of
$\{ \gamma_1,\gamma_1+\gamma_2,\dots \gamma_1+\cdots+\gamma_{k-1} \}$.
It is clear that the lengths of the maximal sequences of consecutive integers in
$I_\gamma$ are $\gamma_1-1,\gamma_2-1,\dots, \gamma_k-1$.
Therefore, $\gamma_1 ! \gamma_2 ! \cdots \gamma_k ! = n_{I_\gamma}$, and hence,
$$
\sum_{I \subseteq [n-1]} \frac{1}{ n_I} = \sum_{k\geq0} \sum_{\stackrel{i_1+\cdots + i_k=n}{i_j \geq 1}} \frac{1}{ i_1! i_2 ! \cdots i_k!}.
$$
Multiplying both sides by $(n!)^2$ gives the desired result that $\ds{n! b_n = b (\ms{X}_e)}$.

\end{proof}

The proof of Theorem \ref{T:final} follows from Lemma \ref{L:thelemma} and comparing the
asymptotics (\ref{A:asymptoticBell}) and (\ref{A:asymptoticFibonacci}) of the ordered Bell numbers and Fibonacci numbers,
respectively.

\end{proof}

\bibliography{Hermite}
\bibliographystyle{plain}


\begin{figure}[htp]
\centering

\begin{tikzpicture}[scale=.56]

\node [shape=circle,draw,fill=blue!10!] at (0,0) (a) {$\emptyset,T$};

\node [shape=circle,draw,fill=purple] at (-6,5) (b1) {$\{\mathbf{1}\},T$};
\node [shape=circle,draw,fill=blue!30!] at (-2,5) (b2) {$\{\mathbf{2}\},T$};
\node [shape=circle,draw,fill=red!20!] at (2,5) (b3) {$\emptyset,T$};
\node [shape=circle,draw,fill=orange!50!] at (6,5) (b4) {$\emptyset,T$};

\node [shape=circle,draw,fill=blue!50!] at (-10,10) (c1) {$\{\mathbf{1}\},T$};
\node [shape=circle,draw,fill=red!20!]  at (-6,10) (c2) {$\{\mathbf{1}\}$};
\node [shape=circle,draw,fill=green!50!] at (-2,10) (c3) {$\{\mathbf{2}\},T$};
\node [shape=circle,draw,fill=orange!50!] at (2,10) (c4) {$\{\mathbf{2}\}$};
\node [shape=circle,draw,fill=red!50!] at (6,10) (c5) {$\emptyset,T$};
\node [shape=circle,draw,fill=yellow!50!] at (10,10) (c6) {$\emptyset,T$};

\node [shape=circle,draw,fill=green!50!] at (-10,15) (d1) {$\{\mathbf{1,2}\}$};
\node  [shape=circle,draw,fill=brown]  at (-6,15) (d2) {$\{\mathbf{1}\},T$};
\node [shape=circle,draw,fill=red!50!] at (-2,15) (d3) {$\{\mathbf{1}\}$};
\node  [shape=circle,draw,fill=green!15]  at (2,15) (d4) {$\{\mathbf{2}\},T$};
\node [shape=circle,draw,fill=yellow!50!] at (6,15) (d5) {$\{\mathbf{2}\}$};
\node  [shape=circle,draw,fill=yellow]  at (10,15) (d6) {$\mathbf{\emptyset},T$};

\node  [shape=circle,draw,fill=brown] at (-6,20) (e1) {$\{\mathbf{1,2}\}$};
\node  [shape=circle,draw,fill=green!15] at (-2,20) (e2) {$\{\mathbf{1,2}\}$};
\node  [shape=circle,draw,fill=yellow]  at (2,20) (e3) {$\{\mathbf{1}\}$};
\node  [shape=circle,draw,fill=yellow]  at (6,20) (e4) {$\{\mathbf{2}\}$};

\node [shape=circle,draw,fill=yellow] at (0,25) (f) {$\{\mathbf{1,2}\}$};

\filldraw[dashed, thick,fill=red] (a) to (b1);
\draw[dashed] (a) to (b2);
\draw[dashed] (a) to (b3);
\draw[dashed] (a) to (b4);

\draw[dashed] (b1) to (c1);
\draw[dashed] (b1) to (c2);

\draw[dashed] (b2) to (c3);
\draw[dashed] (b2) to (c4);

\draw[-, ultra thick] (b3) to (c2);
\draw[dashed] (b3) to (c3);
\draw[dashed] (b3) to (c5);
\draw[dashed] (b3) to (c6);

\draw[dashed] (b4) to (c1);
\draw[-,ultra thick] (b4) to (c4);
\draw[dashed] (b4) to (c5);
\draw[dashed] (b4) to (c6);

\draw[dashed] (c1) to (d1);
\draw[dashed] (c1) to (d2);
\draw[dashed] (c1) to (d3);

\draw[dashed] (c2) to (d2);
\draw[dashed] (c2) to (d3);

\draw[-,ultra thick] (c3) to (d1);
\draw[dashed] (c3) to (d4);
\draw[dashed] (c3) to (d5);

\draw[dashed] (c4) to (d1);
\draw[dashed] (c4) to (d4);
\draw[dashed] (c4) to (d5);

\draw[ultra thick] (c5) to (d3);
\draw[dashed] (c5) to (d4);
\draw[dashed] (c5) to (d6);

\draw[ultra thick] (c6) to (d5);
\draw[dashed] (c6) to (d2);
\draw[dashed] (c6) to (d6);

\draw[dashed] (d1) to (e1);
\draw[dashed] (d1) to (e2);

\draw[dashed] (d1) to (c2);

\draw[ultra thick] (d2) to (e1);
\draw[dashed] (d2) to (e3);

\draw[dashed] (d3) to (e2);
\draw[dashed] (d3) to (e3);

\draw[ultra thick] (d4) to (e2);
\draw[dashed] (d4) to (e4);

\draw[dashed] (d5) to (e1);
\draw[dashed] (d5) to (e4);

\draw[ultra thick] (d6) to (e3);
\draw[ultra thick] (d6) to (e4);

\draw[dashed] (e1) to (f);
\draw[dashed] (e2) to (f);
\draw[ultra thick] (e3) to (f);
\draw[ultra  thick] (e4) to (f);

\end{tikzpicture}
\caption{Cell decomposition and $B$-orbits of the complete quadrics for $n=3$.}
\label{fig:Cells}
\end{figure}
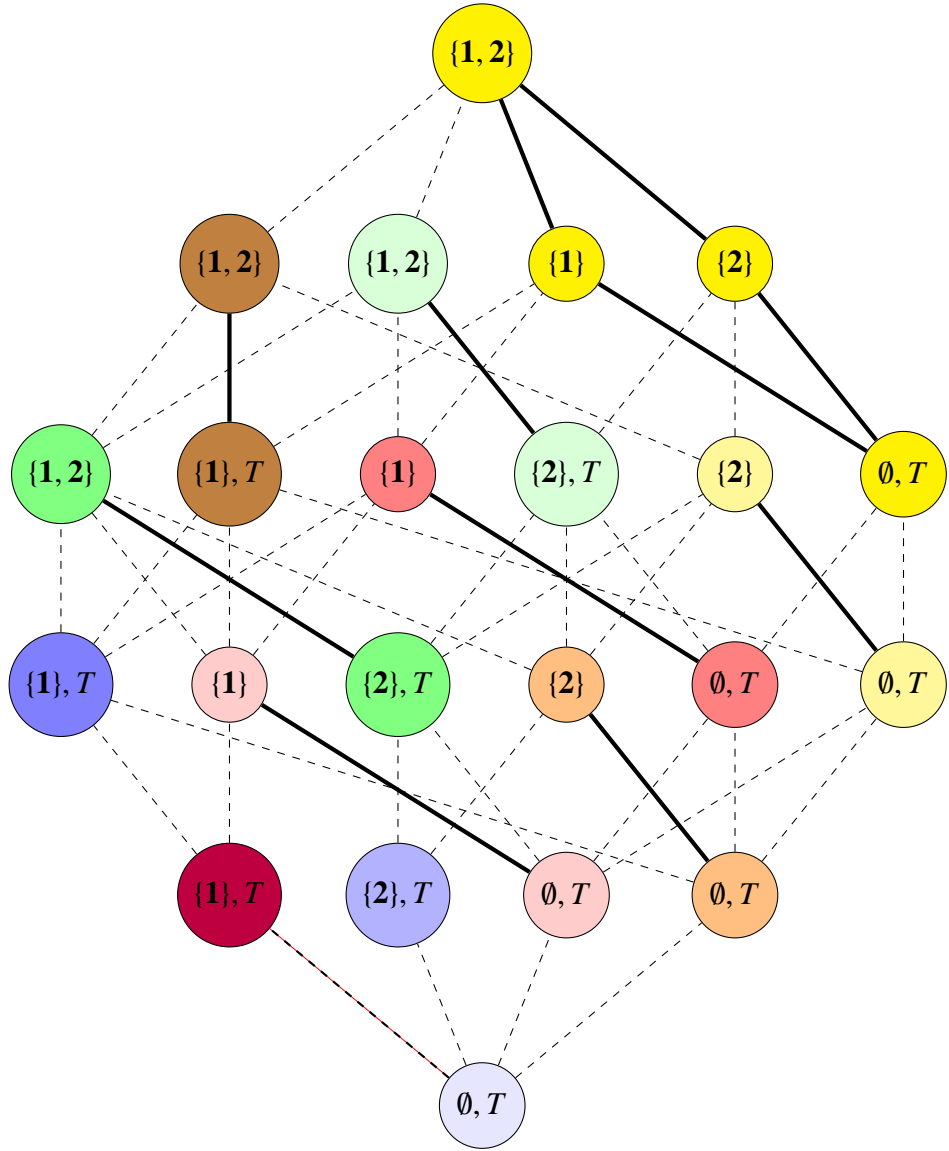

\begin{table}[htp]
\begin{center}
\begin{tabular}{c|c}\label{Table}
$J\subseteq [4]$ & $B_{5,J}(y)$ \\
\hline
$\emptyset$ & $q^4+ 26q^3+66q^2+26q+1$ \\
$\{1\}$ & $q^3+22q^2+33q+4$ \\
$\{2\}$ & $2q^3+29q^2+26q+3$ \\
$\{3\}$ & $3q^3+26q^2+29q+2$ \\
$\{4\}$ & $4q^3+33q^2+22q+1$ \\
$\{1,2\}$ & $3q^2+14q+3$ \\
$\{1,3\}$ & $3q^2+19q+8$ \\
$\{1,4\}$ & $4q^2+22q+4$ \\
$\{2,3\}$ & $7q^2+11q+2$ \\
$\{2,4\}$ & $8q^2+19q+3$ \\
$\{3,4\}$ & $q^2+13q+1$ \\

$\{1,2,3\}$ & $2q+3$ \\
$\{1,2,4\}$ & $7q+3$ \\
$\{1,3,4\}$ & $6q+4$ \\
$\{2,3,4\}$ & $4q+1$ \\

$\{1,2,3,4\}$ & $1$ \\
\end{tabular}
\caption{Generalized $B$-orbit enumerators.}\label{Table}
\end{center}
\end{table}

\end{document}